%% file: main.tex
\documentclass{article}
\usepackage[utf8]{inputenc}

\title{Bounds of the spectral radius of the induced map on cohomology}
\author{Sven Sandfeldt \\ \href{mailto:svensan@kth.se}{svensan@kth.se}}

\date{June 2021}

\usepackage[fleqn]{amsmath}
\usepackage{mathtools}
\usepackage{fancyhdr}
\usepackage{tensor}
\usepackage{amsfonts}
\usepackage{mathrsfs}
\usepackage{amssymb}
\usepackage{graphicx}
\usepackage[shortlabels]{enumitem}
\usepackage{mathabx}
\usepackage{chngcntr}
\usepackage{tikz}
\usepackage{tikz-cd} 
\usepackage{amsthm}
\usepackage{todonotes}
\usepackage[T1]{fontenc}
\usepackage{titling}
\usepackage{afterpage}
\usepackage{hyperref}

\renewcommand{\baselinestretch}{1.10}

\newcommand{\norm}[1]{\left\lVert#1\right\rVert}

\newcommand{\intd}{\text{d}}

\makeatletter
\newtheorem*{rep@theorem}{\rep@title}
\newcommand{\newreptheorem}[2]{%
\newenvironment{rep#1}[1]{%
 \def\rep@title{#2 \ref{##1}}%
 \begin{rep@theorem}}%
 {\end{rep@theorem}}}
\makeatother

\counterwithin{equation}{section}

\newtheorem{mainTheorem}{Theorem}

\newreptheorem{mainTheorem}{Theorem}

\newtheorem{mainCorollary}{Corollary}

\newreptheorem{mainCorollary}{Corollary}

\newtheorem{theorem}{Theorem}[section]

\newtheorem{lemma}{Lemma}[section]
\newtheorem{definition}{Definition}[section]
\newtheorem*{definition*}{Definition}

\theoremstyle{definition}

\newtheorem{remark}{Remark}

\begin{document}

\renewcommand{\baselinestretch}{1.0} 

\maketitle
\begin{abstract}
In this paper we study the relationship between Lyapunov exponents and the induced map on cohomology for $C^{1}-$diffeomorphisms on compact manifolds. We show that if the induced map on cohomology has spectral radius strictly larger than 1, then the diffeomorphism has an invariant ergodic measure with at least one positive Lyapunov exponent. Furthermore, if the diffeomorphism also preserves a continuous volume form then it has an invariant ergodic measure with at least one positive and one negative Lyapunov exponent, in agreement with Shub's entropy conjecture. We also consider diffeomorphisms preserving a measure equivalent to volume. In this case we show that if the Lyapunov metric satisfies an integrability condition then volume must be a measure of maximal entropy.
\end{abstract}

\section{Introduction}

\input{Introduction.tex}



\section{Main results}

\input{MainResults.tex}

\section{Preliminaries and notation}

\input{GeneralSetUp.tex}

\subsection{Lyapunov exponents}

\input{LyapunovExponents.tex}

\subsection{Cohomology and Hodge decomposition}

\input{CohomologyAndHodgeDecomposition.tex}

\section{Proof of main results}

\input{GeneralProofs.tex}

\subsection{Proof of Theorem A}

\input{ProofOfTheoremAndCorollaryA.tex}

\subsection{Proof of Theorem B}

\input{ProofOfTheoremAndCorollaryB.tex}

\section{Proof of Corollaries}

\input{ProofOfCorollaries.tex}

\newpage

\bibliography{main.bib}
\bibliographystyle{abbrv}

\end{document}

%% file: Introduction.tex
The Lyapunov exponents of a general linear cocycle are difficult to evaluate. One reason for this difficulty is that calculating the Lyapunov exponents of a linear cocycle requires knowledge of the values of the cocycle for all future times, and this information is in general not available. Besides, even when it is available, the calculations involved are not tractable. On the other hand, in the special case of the derivative cocycle, the existence of non-zero Lyapunov exponents can have strong implications for the dynamics of a diffeomorphism, see for example \cite{Katok1} for hyperbolic measure. This motivates the search for estimates of the Lyapunov exponents with quantities that are easier to evaluate.

It follows from the Ruelle inequality, see \cite{LectureNotesPollicott}, that a lower bound for the sum of positive Lyapunov exponents is given by the metric entropy. More explicitly, let $f\in\text{Diff}^{1}(X)$ be a $C^{1}-$diffeomorphism of a compact smooth manifold $X$. Given some $f-$invariant ergodic Borel probability measure $\mu\in\mathcal{M}(X)$ we have the inequality
\begin{align*}
h_{\mu}(f)\leq\sum_{\lambda_{i}(x,Df,\mu) > 0}\lambda_{i}(x,Df,\mu),\quad \mu-\text{a.e}
\end{align*}
where $h_{\mu}(f)$ is the metric entropy of $f$ with respect to $\mu$ and $\lambda_{i}(x,Df,\mu)$ are the Lyapunov exponents of the derivative cocycle of $f$ with respect to $\mu$. It follows in particular that if $h_{\mu}(f) > 0$ then there must exist at least one positive Lyapunov exponent. Moreover, using the fact that $h_{\mu}(f^{-1}) = h_{\mu}(f)$ the existence of one negative Lyapunov exponent also follows. By the variational principle, see \cite{KatokIntroDynSyst}, there is a sufficient condition for the existence of a measure with positive metric entropy. If $h_{\text{top}}(f)$ denotes the topological entropy of $f\in\text{Diff}^{1}(X)$ then the variational principle states
\begin{align*}
h_{\text{top}}(f) = \sup_{\mu}h_{\mu}(f)
\end{align*}
where the supremum is over all ergodic measures. It follows that if $h_{\text{top}}(f) > 0$ then there has to exist at least one ergodic measure $\mu$ such that $h_{\mu}(f) > 0$ and hence there exists a measure with at least one positive and one negative Lyapunov exponent. So to find a sufficient condition for the existence of non-zero Lyapunov exponents it suffices to find a sufficient condition for the topological entropy to be positive.

There has been a lot of work on finding lower bounds of the topological entropy. Notably, it was shown by Misiurewicz and Przytycki, \cite{MisiurewiczAndPrzytycki1}, that for a $C^{1}-$map the logarithm of the degree is a lower bound for the topological entropy. That is, let $f\in\text{End}^{1}(X)$ be orientation preserving, then
\begin{align*}
\log(\deg(f))\leq h_{\text{top}}(f)
\end{align*}
where $\deg(f)$ is the topological degree of $f$. It should be noted that since $\deg(f)$ is the eigenvalue of the induced map of $f$ on the top homology group, one can interpret the result as follows: the topological entropy is an upper bound for the homological growth. It was shown by Manning, \cite{Manning1}, that for $f\in\text{Diff}^{1}(X)$ the topological entropy is an upper bound for the logarithm of the spectral radius of the induced map on the first homology. That is, if $f_{*,1}:H_{1}(X;\mathbb{R})\to H_{1}(X;\mathbb{R})$ is the induced map on the first homology group and $\text{sp}(f_{*,1})$ is the spectral radius, then
\begin{align*}
\log(\text{sp}(f_{*,1}))\leq h_{\text{top}}(f)
\end{align*}
where again this result can be interpreted as the topological entropy being an upper bound for the homological growth. The result of Manning has also been generalized by Bowen to the induced map on the fundamental group, see \cite{Bowen1}. It was conjectured by Shub, see \cite{Shub1}, that the results in \cite{MisiurewiczAndPrzytycki1} and \cite{Manning1} was part of a more general principle. Namely, the topological entropy is a upper bound for homological growth. More concretely, let $f:X\to X$ be a $C^{1}-$map and let $f_{*}$ be the induced map on the real homology groups of $X$, that is 
\begin{align*}
& f_{*}:\bigoplus_{k = 0}^{n}H_{k}(X;\mathbb{R})\to\bigoplus_{k = 0}^{n}H_{k}(X;\mathbb{R}),\\
& f_{*}|_{H_{k}(X;\mathbb{R})} = f_{*,k}:H_{k}(X;\mathbb{R})\to H_{k}(X;\mathbb{R})
\end{align*}
where $\dim(X) = n$. Then Shub conjectured that the bound
\begin{align*}
\log\text{sp}(f_{*})\leq h_{\text{top}}(f)
\end{align*}
should hold. This conjecture is commonly known as Shubs entropy conjecture or simply the entropy conjecture. The entropy conjecture is sharp in the sense that there exist Lipschitz homeomorphisms $f\in\text{Homeo}(X)$ such that $f$ does not satisfy the entropy conjecture, see \cite{Shub1}. It follows that the differentiability should be crucial in a proof of the entropy conjecture. There are partial results on the entropy conjecture. Notably, the result of Manning \cite{Manning1} combined with Poincaré duality proves the entropy conjecture for all manifolds of dimension at most $3$. In \cite{Yomdin1} Yomdin shows that the entropy conjecture holds for $C^{\infty}-$maps. Actually, Yomdin's result is stronger than the entropy conjecture in that he shows that the topological entropy is an upper bound for the volume growth. And the volume growth, in turn, is larger than the homological growth. There also exist partial results on the entropy conjecture by restricting the type of manifold considered. In \cite{Marzantowicz1} it is shown that the entropy conjecture holds for every continuous map on a nilmanifold.

The main result of this paper is that some of the consequences of the entropy conjecture still hold without the full conjecture. In particular, we apply a variational principle from \cite{SchreiberSubadditiveFunction} to show that any diffeomorphism $f:X\to X$ with spectral radius larger than one has at least one ergodic measure with a positive Lyapunov exponent, see Corollary C. More precisely, we show that there is some ergodic measure $\mu$ such that the following inequality holds
\begin{align}
\label{Eq1}
\log\text{sp}(f_{*})\leq\Sigma(x,Df,\mu),\quad\mu-a.e
\end{align}
where $\Sigma(x,Df,\mu)$ is the sum of positive Lyapunov exponents with respect to $\mu$. We can, however, not guarantee the existence of an ergodic measure with at least one positive and one negative exponent. On the other hand, if the diffeomorphism $f$ also preserves a continuous volume form, then the sequence of determinants $\det(Df^{n})$ is uniformly bounded, therefore by Oseledec's theorem the Lyapunov exponents of every ergodic measure must sum to zero. It follows from our results that if a diffeomorphism has spectral radius larger then $1$ and preserves a continuous volume, then there is an ergodic measure with at least one positive and one negative Lyapunov exponent.

A natural question in light of $\eqref{Eq1}$ is for what measures do we obtain the inequality $\eqref{Eq1}$. In particular, if $f$ preserves a volume $\intd V$, under what conditions can we obtain the inequality 
\begin{align*}
\log\text{sp}(f_{*})\leq\Sigma(x,Df,\mu)
\end{align*}
for $\intd\mu = \intd V$? We show that this is possible provided that the Lyapunov metric satisfies an integrability condition, see Corollary D. Using Pesin's entropy formula this also gives a positive answer to the entropy conjecture for conservative diffeomorphisms where the Lyapunov metric satisfies an integrability condition. Actually combining our results with the results of \cite{KozlovskiIntForm} we also show that in the $C^{\infty}-$setting the integrability condition from Corollary D also implies that the volume is a measure of maximal entropy for $C^{\infty}-$diffeomorphisms.

\textbf{Structure of paper:} In section $2$ we formulate the main result of the paper and briefly discuss the proofs. In section $3$ we go through some background from smooth ergodic theory, differential topology and Hodge theory and simultaneously fix notation. In section $4$ we prove some technical results which are used to prove the obtain the main results of the paper. In section $5$ we prove corollaries of the results of section $4$. 

\textbf{Acknowledgement:} This research has received support from the Swedish Research Council grant 2019-04641. I would like to thank Danijela Damjanovic for providing support during the writing of this paper.

%% file: MainResults.tex
In this section we state the main results of the paper. The aim is to obtain bounds for the spectral radius in terms of Lyapunov exponents.

Let $(X,g)$ be a smooth, oriented, compact Riemannian manifold without boundary and metric $g$. We denote by $V_{g}$ the volume form induced by the metric $g$, we shall always assume that $g$ is chosen such that $V_{g}(X) = 1$. Let $f:X\to X$ be a $C^{1}-$map. We denote by $H^{k}(f):H^{k}(X)\to H^{k}(X)$ the induced map on the $k'$th real cohomology group, or equivalently the $k'$th de Rahm cohomology group. Let $\Omega_{\mathbb{C}}^{k}(X)$ be the space of complex smooth $k-$forms over $X$ and let
\begin{align*}
\intd:\Omega_{\mathbb{C}}^{k}(X)\to\Omega_{\mathbb{C}}^{k+1}(X)
\end{align*}
denote the exterior derivative. If $[\omega]\in H^{k}(X)\otimes\mathbb{C}$ is an eigenvector for $H^{k}(f)$ with eigenvalue $e^{\lambda}\in\mathbb{C}$ then there is a harmonic $k-$form $\omega$ and a continuous $k-$form $\alpha\in\overline{\text{Im}(\intd)}$ such that
\begin{align*}
\tag{Eq$_{\lambda,k}$}
\label{MainEquationIntroduction}
f^{*}\omega = e^{\lambda}\omega + \alpha
\end{align*}
from Lemma $3.1$. We define
\begin{definition}
We say that $\omega,\alpha\in L^{2}\Omega_{\mathbb{C}}^{k}(X)$ is a solution of $\eqref{MainEquationIntroduction}$ if $\omega$ and $\alpha$ satisfy the equation in $\eqref{MainEquationIntroduction}$ and if $\omega\in\mathcal{H}^{k}$, $\alpha\in\overline{\text{Im}(\intd)}$.
\end{definition}
with this definition there is a bijective correspondence of the non-trivial solutions to $\eqref{MainEquationIntroduction}$ and the eigenvalues of $H^{k}(f)$, see Lemma $3.1$.

We define
\begin{align*}
\lambda^{+}(Df^{\wedge k}) = \lim_{n\to\infty}\sup_{x\in X}\frac{1}{n}\log\norm{D_x\left(f^{n}\right)^{\wedge k}}
\end{align*}
Our first result concerning the spectral radius is essentially the elementary bound of the volume growth, but we state it as a theorem since it will be important in the remainder.
\begin{mainTheorem}
Let $f:X\to X$ be a $C^{1}-$diffeomorphism and let $k$ be an integer between $1$ and $\dim(X)$. If $\omega,\alpha\in L^{2}\Omega_{\mathbb{C}}^{k}(X)$ is a non-trivial solution of $\eqref{MainEquationIntroduction}$, then $\text{Re}(\lambda)\leq\lambda^{+}(Df^{\wedge k})$.
\end{mainTheorem}
As an immediate consequence of Theorem A we can consider the case of uniformly subexponential maps $f:X\to X$. We say that a $C^{1}-$diffeomorphism $f:X\to X$ is \textit{uniformly subexponential} if every Lyapunov exponent with respect to every invariant measure is $0$. Equivalently $f$ is uniformly subexponential if $\lambda^{+}(Df^{\wedge k}) = 0$ for every $k$. So we obtain the following
\begin{mainCorollary}
If $f:X\to X$ is a uniformly subexponential $C^{1}-$diffeomorphism then 
\begin{align*}
\log\text{sp}(f_{*}) = 0.
\end{align*}
\end{mainCorollary}
Let $\lambda^{+}(Df^{\wedge k}),\mu)$ be the average maximal Lyapunov exponent of $Df^{\wedge k}$ with respect to an invariant measure $\mu$ defined by
\begin{align*}
\lambda^{+}(Df^{\wedge k},\mu) = \lim_{n\to\infty}\frac{1}{n}\int_{X}\log\norm{\left(D_{x}f^{n}\right)^{\wedge k}}\intd\mu
\end{align*}
and let $\Lambda_{k}(Df,\mu)$ be the sum of the $k$ largest Lyapunov exponents (counting with multiplicity) with respect to the measure $\mu$. Using the results of \cite{SchreiberSubadditiveFunction}, $\lambda^{+}(x,Df^{\wedge k},\mu) = \Lambda_{k}(x,Df,\mu)$ and the fact that
\begin{align*}
\mu\mapsto\lambda^{+}(Df^{\wedge k},\mu)
\end{align*}
is upper semi-continuous we obtain the following corollary
\begin{mainCorollary}
Let $f:X\to X$ be a $C^{1}-$diffeomorphism and let $k$ be an integer between $1$ and $\dim(X)$. If $\omega,\alpha\in L^{2}\Omega_{\mathbb{C}}^{k}(X)$ is a non-trivial solution to $\eqref{MainEquationIntroduction}$, or equivalently if $e^{\lambda}$ is an eigenvalue for $H^{k}(f)$, then there exist some invariant measure $\nu_{k}\in\mathcal{M}_{\text{erg}}(X)$ such that
\begin{align*}
\text{Re}(\lambda)\leq\Lambda_{k}(Df,\nu_{k})
\end{align*}
in particular it holds that
\begin{align*}
\log\text{sp}(f_{*}) = \log\text{sp}(H^{*}(f))\leq\Sigma(Df,\nu)
\end{align*}
for some $\nu\in\mathcal{M}_{\text{erg}}(X)$.
\end{mainCorollary}
As a consequence of Corollary B we have that if $\text{sp}(f_{*}) > 1$ then there is some measure $\nu\in\mathcal{M}_{\text{erg}}(X)$ with at least one positive Lyapunov exponent. If we add the assumption that $f:X\to X$ preserves a continuous volume form then for every ergodic measure $\mu$ we have
\begin{align*}
\sum_{i = 1}^{\dim(X)}\lambda_{i}(Df,\mu) = 0,\quad \lambda_{i}(Df,\mu) := \int_{X}\lambda_{i}(x,Df,\mu)\intd\mu(x)
\end{align*}
that is, the sum of Lyapunov exponents vanishes and we obtain the following Corollary
\begin{mainCorollary}
If $f:X\to X$ is a $C^{1}-$diffeomorphism with $\text{sp}(f_{*}) > 1$ then $f$ has a invariant ergodic measure with at least one positive Lyapunov exponent. Furthermore, if $f$ also preserves a continuous volume form then $f$ has a invariant ergodic measure with at least one positive and one negative Lyapunov exponent.
\end{mainCorollary}
For any metric vector bundle $\mathcal{E}\to X$ with metric $h$ we can define the space of $L^{p}-$sections as the sections $\sigma:X\to\mathcal{E}$ such that
\begin{align*}
\norm{\sigma}_{L^{p}}^{p} = \int_{X}\norm{\sigma(x)}_{h}^{p}\intd V_{g}(x) < \infty
\end{align*}
where $\norm{\cdot}_{h}$ is the norm induced by $h$. We also allow $p < 1$, even though in this case the integral above does not necessarily define norm. In particular any bundle
\begin{align*}
\text{T}_{r}^{s}X = \left(\text{T}X\right)^{\otimes s}\otimes\left(\text{T}^{*}X\right)^{\otimes r},\quad\Lambda^{k}(\text{T}X),\quad\Lambda^{k}(\text{T}^{*}X)
\end{align*}
can naturally be given an $L^{p}-$structure by the Riemannian metric on $X$. Any metric $h$ on $X$ defines a section $h:X\to\text{T}_{2}^{0}X = \text{T}^{*}X\otimes\text{T}^{*}X$. We say that $h$ is a $L^{p}-$metric if $\norm{h}_{L^{p}} < \infty$. Let $f:X\to X$ be a $C^{1}-$diffeomorphism preserving a measure $V$ equivalent to $V_{g}$. We denote by $\lambda_{i}(x,Df,V)$ the $i'$th Lyapunov exponent of $Df$ with respect to $V$ counted with multiplicity. Let $\Tilde{\lambda}_{i}(x,Df,V)$ be the $i'$th Lyapunov exponent counted without multiplicity. For $V-$almost every $x\in X$ we define the Lyapunov splitting $H_{i}(x)\subset T_{x}X$ defined by
\begin{align*}
\lim_{n\to\infty}\frac{1}{n}\log\norm{D_{x}f^{n}(v)} = \Tilde{\lambda}_{i}(x,Df,V),\quad v\in H_{i}(x)\setminus\{0\}.
\end{align*}
We define the (family of) Lyapunov metrics, see \cite{KatokIntroDynSyst}, on $H_{i}(x)$ by
\begin{align*}
h_{i}^{\varepsilon} := \sum_{n\in\mathbb{Z}}e^{-2|n|\varepsilon}e^{-2n\Tilde{\lambda}_{i}(x,Df,V)}\left(f^{n}\right)^{*}g 
\end{align*}
where $\left(f^{n}\right)^{*}g$ is the pullback of $g$
\begin{align*}
\left(f^{n}\right)^{*}g_{x}(u,v) = g_{f^{n}x}\left(D_{x}f^{n}(u),D_{x}f^{n}(v)\right).
\end{align*}
We can extend this to a metric on all of $\text{T}_{x}X$ be defining the inner product between $u\in H_{i}(x)$ and $v\in H_{j}(x)$ to be $0$ for $i\neq j$. That is we define
\begin{align*}
h^{\varepsilon} := \sum_{i}h_{i}^{\varepsilon}
\end{align*}
This defines a measurable $V_{g}-$almost everywhere defined metric. We can now state our second main result
\begin{mainTheorem}
Let $k$ be an integer between $1$ and $\dim(X)$ and let $f:X\to X$ be a $C^{1}-$diffeomorphism preserving a measure $V$ equivalent to the Riemannian volume. If $h^{\varepsilon}$ is $L^{k/2}$ and $\omega,\alpha\in L^{2}\Omega_{\mathbb{C}}^{k}(X)$ is a non-trivial solution to $\eqref{MainEquationIntroduction}$ then
\begin{align*}
\text{Re}(\lambda)\leq\norm{\Lambda_{k}(x,Df,V)}_{L^{\infty}} + k\varepsilon.
\end{align*}
\end{mainTheorem}
Here $\norm{\sigma}_{L^{\infty}}$ is the essential supremum of the function $\sigma:X\to\mathbb{C}$ with respect to volume. By using the universal coefficients theorem we obtain the following corollary
\begin{mainCorollary}
Let $f:X\to X$ be a $C^{1}-$diffeomorphism preserving a measure $V$ equivalent to the Riemannian volume. If $h^{\varepsilon}$ is $L^{\dim(X)/2}$ then
\begin{align*}
\log\text{sp}(f_{*})\leq\norm{\Sigma(x,Df,V)}_{L^{\infty}} + \dim(X)\varepsilon.
\end{align*}
So in particular if $\text{sp}(f_{*}) > 1$ and $\varepsilon$ is sufficiently small then there exists a set of positive volume where $f$ has at least one positive Lyapunov exponent and one negative Lyapunov exponent.
\end{mainCorollary}
In the extreme case where $h^{\varepsilon}$ is $L^{\dim(X)/2}$ for every $\varepsilon>0$ and $f$ is ergodic we can use Pesin's entropy formula to prove Shub's entropy conjecture in this case.
\begin{mainCorollary}
If $f:X\to X$ is a conservative and ergodic  $C^{1+\alpha}-$diffeomorphism with $h^{\varepsilon}$ in $L^{\dim(X)/2}$ for every $\varepsilon > 0$ then $f$ satisfies Shub's entropy conjecture.
\end{mainCorollary}
\begin{remark}
We note that the conclusion of the Corollary is stronger then Shub's entropy conjecture since we actually prove
\begin{align*}
\log\text{sp}(f_{*})\leq\Sigma(Df) = h_{V}(f)\leq h_{\text{top}}(f)
\end{align*}
where the last equality use Pesin's entropy formula. We also note that we only need the $C^{1+\alpha}$ assumption to be able to apply Pesin's entropy formula, so the corollary holds whenever the system satisfies Pesin's entropy formula.
\end{remark}
Actually by analysing the proof of Theorem B we have
\begin{align*}
\liminf_{n\to\infty}\frac{1}{n}\log\int_{X}\norm{\left(D_{x}f^{n}\right)^{\wedge}}_{g}\intd V_{g}\leq\norm{\Sigma(x;Df)}_{L^{\infty}} + \dim(X)\varepsilon
\end{align*}
if $h^{\varepsilon}$ is $L^{\dim(X)/2}$. Using the main result from \cite{KozlovskiIntForm} we have
\begin{align*}
h_{\text{top}}(f) = \lim_{n\to\infty}\frac{1}{n}\log\int_{X}\norm{\left(D_{x}f^{n}\right)^{\wedge}}\intd V_{g}(x)
\end{align*}
for $C^{\infty}-$diffeomorphisms. So in particular, if $f:X\to X$ is a conservative ergodic $C^{\infty}-$diffeomorphism (or more generally, we only need the Lyapunov exponents to be constant almost everywhere) such that $h^{\varepsilon}$ is $L^{\dim(X)/2}$ for every $\varepsilon>0$ then
\begin{align*}
h_{\text{top}}(f) = \lim_{n\to\infty}\frac{1}{n}\log\int_{X}\norm{\left(D_{x}f^{n}\right)^{\wedge}}\intd V_{g}(x)\leq\Sigma(Df) = h_{V}(f)\leq h_{\text{top}}(f)
\end{align*}
where we've used Pesin's entropy formula and the variational principle. So under these assumptions the volume $V$ must be a measure of maximal entropy.
\begin{mainCorollary}
If $f:X\to X$ is a conservative, ergodic $C^{\infty}-$diffeomorphism with $h^{\varepsilon}\in L^{\dim(X)/2}$ for every $\varepsilon>0$, then $V$ is a measure of maximal entropy for $f$.
\end{mainCorollary}

%% file: GeneralSetUp.tex
Let $(X,g)$ be a smooth, compact, oriented Riemannian manifold without boundary. We will consider a $C^{1}-$diffeomorphism
\begin{align*}
f:X\to X    
\end{align*}
 which will be assumed fixed for the remainder of this section. We denote by $\mathcal{M}(X)$ the space of $f-$invariant Borel probability measures on $X$. We denote by $\mathcal{M}_{\text{erg}}(X)$ the space of $f-$invariant ergodic Borel probability measures on $X$.
 
Let $\pi_{\mathcal{E}}:\mathcal{E}\to X$ be a continuous metric (possibly complex, in which case the metric on $\mathcal{E}$ is assumed to be hermitian) finite rank vector bundle over $X$. We say that a map
\begin{align*}
\Phi:\mathbb{Z}\times\mathcal{E}\to\mathcal{E},\quad\text{denoted }\Phi(n,x)v,\quad x\in X,\text{ }v\in\mathcal{E}_{x},\text{ }n\in\mathbb{Z}
\end{align*}
is a linear cocycle over $f:X\to X$ if it holds that
\begin{align*}
\pi_{\mathcal{E}}\Phi(n,x)v = f^{n}x,\quad x\in X,\text{ }v\in\mathcal{E}_{x},\text{ }n\in\mathbb{Z}
\end{align*}
and if $\Phi(n,x):\mathcal{E}_{x}\to\mathcal{E}_{f^{n}x}$ is linear and satisfy the cocycle equation
\begin{align*}
\Phi(n+m,x)v = \Phi(n,f^{m}x)\Phi(m,x)v,\quad x\in X,\text{ }v\in\mathcal{E}_{x},\text{ }n,m\in\mathbb{Z}.
\end{align*}
If $\mathcal{E}$ is a complex vector bundle we also require that $\Phi(x,n)$ is complex linear.

Let $h:X\to\mathcal{E}^{*}\otimes\mathcal{E}^{*}$ denote the metric on $\mathcal{E}$. We can define a norm of a cocycle $\Phi$ at $x\in X$ as the operator norm of $\Phi(1,x)$
\begin{align*}
\norm{\Phi}_{x} = \sup_{0\neq\nu\in\mathcal{E}_{x}}\frac{\norm{\Phi(1,x)\nu}}{\norm{\nu}}
\end{align*}
where the norm on $\mathcal{E}_{x}$ and $\mathcal{E}_{fx}$ is the norm induced by $h$. We note that if $\Phi$ is a continuous cocycle then the map $x\mapsto\norm{\Phi}_{x}$ is continuous and we can define
\begin{align*}
\norm{\Phi} := \sup_{x}\norm{\Phi}_{x} < \infty.
\end{align*}
If $\pi_{\mathcal{E}^{*}}:\mathcal{E}^{*}\to X$ denotes the dual bundle of $\mathcal{E}$, then any cocycle $\Phi$ in $\mathcal{E}$ over $f$ induces a cocycle $\Phi^{*}$ in $\mathcal{E}^{*}$ over $f^{-1}$. We define this dual cocycle
\begin{align*}
\Phi^{*}(n,f^{n}x):\mathcal{E}_{f^{n}x}^{*}\to\mathcal{E}_{x}^{*}
\end{align*}
as the dual map of the map $\Phi(n,x):\mathcal{E}_{x}\to\mathcal{E}_{f^{n}x}$. We note that if $\Phi$ is a continuous cocycle then so is $\Phi^{*}$.

We have two natural norms on the vector bundle $\pi_{\mathcal{E}^{*}}:\mathcal{E}^{*}\to X$. On the one hand we have the operator norm
\begin{align*}
\norm{u} := \sup_{0\neq\nu\in\mathcal{E}_{x}}\frac{|u(\nu)|}{\norm{\nu}},\quad u\in\mathcal{E}_{x}^{*}
\end{align*}
where $\norm{\nu}$ is the norm of $\nu\in\mathcal{E}_{x}$ with respect to the norm induced by the metric $h$. On the other hand we have a (anti-)isomorphism $\mathcal{E}^{*}\to\mathcal{E}$ defined as the inverse of the map
\begin{align*}
\nu\mapsto h(\cdot,\nu).
\end{align*}
We denote this map by
\begin{align*}
\mathcal{E}_{x}^{*}\ni u\mapsto u^{\sharp}\in\mathcal{E}_{x}
\end{align*}
and define a metric on $\mathcal{E}^{*}$ by
\begin{align*}
h^{*}(u,v) = \overline{h(u^{\sharp},v^{\sharp})}
\end{align*}
where $\overline{z}$ is the complex conjugate of $z\in\mathbb{Z}$. This metric also induces a norm on $\mathcal{E}^{*}$. However by standard Hilbert spaces theory these norms coincide, so we can change between them whenever it is convenient.

%% file: LyapunovExponents.tex
For a cocycle $\Phi:\mathbb{Z}\times\mathcal{E}\to\mathcal{E}$ and some $f-$invariant measure $\mu\in\mathcal{M}(X)$ we can define the \textit{Lyapunov exponent} by
\begin{align*}
\lambda(x,\nu,\Phi,\mu) := \lim_{n\to\infty}\frac{1}{n}\log\norm{\Phi(n,x)\nu},\quad x\in X,\text{ }\nu\in\mathcal{E}_{x}
\end{align*}
where the limit exists for $\mu-$almost every $x$ and every $\nu\in\mathcal{E}_{x}$. By Oseledec's theorem we have a measurable splitting at $\mu-$almost every $x\in X$
\begin{align*}
\mathcal{E}_{x} = \bigoplus_{i = 1}^{k(x)}H_{i}(x)
\end{align*}
such that for $\nu\in H_{i}(x)$ we have
\begin{align*}
\lambda(x,\nu,\Phi,\mu) = \lambda_{i}(x,\Phi,\mu).
\end{align*}
We denote by $k(x)$ the number of distinct Lyapunov exponents at $x$ and $u_{i}(x) = \dim(H_{i}(x))$, then $k$ and $u_{i}$ are $f-$invariant measurable functions. In particular $k(x)$ and $u_{i}(x)$ are constant $\mu-$almost everywhere if $\mu$ is ergodic. If the rank of the vector bundle $\mathcal{E}$ is $r$, then counting with multiplicity we define a decreasing sequence of Lyapunov exponents $\lambda_{1}(x,\Phi,\mu)\geq...\geq\lambda_{r}(x,\Phi,\mu)$. We define the \textit{averaged Lyapunov exponents} by
\begin{align*}
\lambda_{i}(\Phi,\mu) := \int_{X}\lambda_{i}(x,\Phi,\mu)\intd\mu(x).
\end{align*}
If the measure $\mu$ is ergodic then $\lambda_{i}(x,\Phi,\mu) = \lambda_{i}(\Phi,\mu)$ $\mu-$almost everywhere, since the Lyapunov exponents of an ergodic measure are constant.

We define the \textit{maximal Lyapunov exponent} of $\Phi$, with respect to $\mu$, as the limit
\begin{align*}
\lambda^{+}(x,\Phi,\mu) := \lim_{n\to\infty}\frac{1}{n}\log\norm{\Phi(n,x)}
\end{align*}
which exists $\mu-$almost everywhere by the subadditive ergodic theorem. To get a Lyapunov exponent independent of $x$ we also define the \textit{averaged maximal Lyapunov exponent} by
\begin{align*}
\lambda^{+}(\Phi,\mu) = \int_{X}\lambda^{+}(x,\Phi,\mu)\intd\mu(x)
\end{align*}
If $\mu$ is ergodic $\lambda^{+}(x,\Phi,\mu) = \lambda^{+}(\Phi,\mu)$ $\mu-$almost everywhere. It can be shown that we have $\lambda^{+}(x,\Phi,\mu) = \lambda_{1}(x,\Phi,\mu)$, see for example \cite{LectureNotesPollicott}.

Given a cocycle $\Phi$ on the vector bundle $\pi_{\mathcal{E}}:\mathcal{E}\to X$ we can define a cocycle on the vector bundle of $k-$vectors
\begin{align*}
\Lambda^{k}(\mathcal{E}) = \mathcal{E}\wedge...\wedge\mathcal{E}
\end{align*}
by the formula
\begin{align*}
\Phi^{\wedge k}(n,x)(v_{1}\wedge...\wedge v_{k}) = (\Phi(n,x)v_{1})\wedge...\wedge(\Phi(n,x)v_{n}).
\end{align*}
Furthermore, see \cite{BeyondUniformBonatti}, we have the following equalities
\begin{align*}
\lambda^{+}(x,\Phi^{\wedge k},\mu) = \sum_{i=1}^{k}\lambda_{i}(x,\Phi,\mu)
\end{align*}
That is $\lambda^{+}(x,\Phi^{\wedge k},\mu)$ is given by the sum of the $k$ largest Lyapunov exponents of $\Phi$. We define
\begin{align*}
\Lambda_{k}(x,\Phi,\mu) := \sum_{i=1}^{k}\lambda_{i}(x,\Phi,\mu),
\end{align*}
and obtain the equality
\begin{align*}
\lambda^{+}(x,\Phi^{\wedge k},\mu) = \Lambda_{k}(x,\Phi,\mu).
\end{align*}
Similarly as above we define
\begin{align*}
\Lambda_{k}(\Phi,\mu) = \int_{X}\Lambda_{k}(x,\Phi,\mu)\intd\mu(x),
\end{align*}
and also get the equality $\Lambda_{k}(\Phi,\mu) = \lambda^{+}(\Phi^{\wedge k},\mu)$. We define the \textit{sum of positive Lyapunov exponents} by
\begin{align*}
\Sigma(x,\Phi,\mu) := \sum_{\lambda_{i}(x,\Phi,\mu) > 0}\lambda_{i}(x,\Phi,\mu),\quad\Sigma(\Phi,\mu) := \int_{X}\Sigma(x,\Phi,\mu)\intd\mu(x)
\end{align*}
which satisfy the inequalities $\Sigma(x,\Phi,\mu)\geq\Lambda_{k}(x,\Phi,\mu)$ and $\Sigma(\Phi,\mu)\geq\Lambda_{k}(\Phi,\mu)$ for all $k$.

Finally to get exponents that are independent of the measure, we make the following definition
\begin{align*}
& \lambda^{+}(\Phi) := \lim_{n\to\infty}\frac{1}{n}\sup_{x\in X}\log\norm{\Phi(n,x)}.
\end{align*}
For every ergodic $\mu$ it's clear that we have the inequality
\begin{align*}
\lambda^{+}(\Phi,\mu)\leq\lambda^{+}(\Phi).
\end{align*}
In the converse direction we have from \cite[Theorem $1$]{SchreiberSubadditiveFunction} the equalities
\begin{align*}
\sup_{x}\limsup_{n\to\infty}\frac{1}{n}\log\norm{\Phi(n,x)} = \lambda^{+}(\Phi) = \sup_{\mu}\lambda^{+}(\Phi,\mu)
\end{align*}
where the supremum in the last equality is over all ergodic $\mu$. We phrase this as a theorem
\begin{theorem}
We have the equality
\begin{align*}
\lambda^{+}(\Phi) = \sup_{\mu}\lambda^{+}(\Phi,\mu)
\end{align*}
where the supremum is over $\mu\in\mathcal{M}_{\text{erg}}^{\varphi}(X)$.
\end{theorem}

%% file: CohomologyAndHodgeDecomposition.tex
Let $H^{k}(X)$ denote the $k'$th singular cohomology group of $X$. Given a continuous map
\begin{align*}
f:X\to Y
\end{align*}
we denote the induced map on cohomology by $H^{k}(f):H^{k}(Y)\to H^{k}(X)$. For any $k$ we also denote by $\Omega^{k}(X) := \Gamma(\Lambda^{k}(\text{T}^{*}X))$ the space of smooth $k-$forms. Given some smooth $f:X\to Y$ we define the pullback $f^{*}:\Omega^{k}(Y)\to\Omega^{k}(X)$ of differential forms by the formula
\begin{align*}
f^{*}\omega_{x}(X_{1},...,X_{k}) := \omega_{fx}(D_{x}f(X_{1}),...,D_{x}f(X_{k})).
\end{align*}
We note that this formula makes sense for $C^{1}-$maps as well. Let 
\begin{align*}
\intd:\Omega^{k}(X)\to\Omega^{k+1}(X)
\end{align*}
denote the exterior derivative. We obtain the de Rahm cohomology groups as
\begin{align*}
H^{k}_{\text{dR}}(X) = \frac{\ker(\intd:\Omega^{k}(X)\to\Omega^{k+1}(X))}{\text{Im}(\intd:\Omega^{k-1}(X)\to\Omega^{k}(X))}.
\end{align*}
The pullback commutes with the differential, so given some smooth map $f:X\to Y$ we obtain a map on cohomology $f^{*}:H_{\text{dR}}^{k}(Y)\to H_{\text{dR}}^{k}(X)$. By de Rahm's theorem we have isomorphisms $H^{k}(X)\to H_{\text{dR}}^{k}(X)$ such that the following diagram commute
\[
\begin{tikzcd}[column sep=5em, every arrow/.append style={-latex}]
H^{k}(Y)\arrow{r}{H^{k}(f)}\arrow{d} & H^{k}(X)\arrow{d} \\
H_{\text{dR}}^{k}(Y)\arrow{r}{f^{*}} & H_{\text{dR}}^{k}(X)
\end{tikzcd}
\]
for some smooth $f:X\to Y$. For the remainder we shall drop the index $\text{dR}$ and simply consider the de Rahm cohomology groups. 

Let $(X,g)$ be a smooth, compact and orientable Riemannian manifold. Furthermore let $f:X\to X$ be a $C^{1}-$diffeomorphism. The Riemannian metric $g$ induces a metric, denoted $g^{k}$, on every bundle $\Lambda^{k}(\text{T}X)$ by defining
\begin{align*}
g^{k}(v_{1}\wedge...\wedge v_{k},w_{1}\wedge...\wedge w_{k}) = \det(g(v_{i},w_{j})),\quad v_{i},w_{j}\in\text{T}_{x}X.
\end{align*}
Since $g$ induces an isomorphism between $\text{T}X$ and $\text{T}^{*}X$ we can also use $g$ to define a metric on $\text{T}^{*}X$ and by the same construction as above we get an inner product, also denoted $g^{k}$, on $\Lambda^{k}(\text{T}^{*}X)$. This induces an inner product on the space $\Omega^{k}(X)$ of $k-$forms by integrating the inner products of two $k-$forms against the Riemannian volume $V_{g}$
\begin{align*}
\langle\omega,\eta\rangle = \int_{X}g_{x}^{k}\left(\omega_{x},\eta_{x}\right)\intd V_{g}(x),\quad\omega,\eta\in\Omega^{k}(X)
\end{align*}
where $g_{x}^{k}\left(\omega_{x},\eta_{x}\right)$ is the inner product between $\omega_{x}$ and $\eta_{x}$. We denote by $\intd^{*}:\Omega^{k+1}(X)\to\Omega^{k}(X)$ the dual of the exterior derivative with respect to this inner product on $\Omega^{k}(X)$. We define the laplacian on $\Omega^{k}(X)$ to be the map defined by
\begin{align*}
\Delta := \intd^{*}\intd + \intd\intd^{*},
\end{align*}
for more about the Laplacian see for example \cite{GeomAnalysis}. We denote by $\mathcal{H}^{k} := \ker(\Delta:\Omega^{k}(X)\to\Omega^{k}(X))$ the space of harmonic $k-$forms. A calculation shows that if $\omega\in\mathcal{H}^{k}$ then
\begin{align*}
0 = \langle\omega,\Delta\omega\rangle = \langle\intd\omega,\intd\omega\rangle + \langle\intd^{*}\omega,\intd^{*}\omega\rangle = \norm{\intd\omega}^{2} + \norm{\intd^{*}\omega}^{2}
\end{align*}
so in particular we have $\intd\omega = 0$ for $\omega\in\mathcal{H}^{k}$, and we can define the quotient map $\mathcal{H}^{k}\to H^{k}(X)$. The Hodge theorem says that the map $\mathcal{H}^{k}\to H^{k}(X)$ is an isomorphism. Furthermore we have the Hodge decomposition
\begin{align*}
\Omega^{k}(X) = \mathcal{H}^{k}\oplus\text{Im}(\intd)\oplus\text{Im}(\intd^{*}).
\end{align*}
Let $L^{2}\Omega^{k}(X)$ be the closure of $\Omega^{k}(X)$ with respect to the inner product induced by $g$. The Hodge decomposition extends to an orthogonal decomposition
\begin{align*}
L^{2}\Omega^{k}(X) = \mathcal{H}_{k}\oplus\overline{\text{Im}(\intd)}\oplus\overline{\text{Im}(\intd^{*})}.
\end{align*}
We note that given a $C^{\infty}-$map $h:X\to X$ we can decompose the map $H^{k}(h):H^{k}(X)\to H^{k}(X)$ as
\begin{align*}
H^{k}(X)\xrightarrow{}\mathcal{H}^{k}\xrightarrow{h^{*}}L^{2}\Omega^{k}(X)\xrightarrow{P}\mathcal{H}^{k}\xrightarrow{}H^{k}(X)
\end{align*}
where $P:L^{2}\Omega^{k}(X)\to\mathcal{H}^{k}$ is the projection map. By approximating a $C^{1}-$map with $C^{\infty}-$maps it follows that this holds for $C^{1}-$maps as well. We have the following lemma
\begin{lemma}
Let $f:X\to X$ be a $C^{1}-$map. Then $H^{k}(f):H^{k}(X)\to H^{k}(X)$ is given by
\begin{align*}
H^{k}(X)\xrightarrow{}\mathcal{H}^{k}\xrightarrow{f^{*}}L^{2}\Omega^{k}(X)\xrightarrow{P}\mathcal{H}^{k}\xrightarrow{}H^{k}(X)
\end{align*}
and if $\omega,\eta\in\mathcal{H}^{k}$ are such that $H^{k}(f)([\omega]) = [\eta]$ then
\begin{align*}
f^{*}\omega = \eta + \alpha
\end{align*}
where $\alpha\in\overline{\text{Im}(\intd)}$ is a continuous section. Furthermore $f^{*}$ preserve $\overline{\ker(\intd)}$ and $\text{Im}(\intd)$.
\end{lemma}
\begin{proof}
We note that the first claim follows from the second since $P(\eta + \alpha) = \eta$. So it suffices to show the formula
\begin{align*}
f^{*}\omega = \eta + \alpha
\end{align*}
for $\omega,\eta\in\mathcal{H}^{k}$ such that $H^{k}(f)([\omega]) = [\eta]$ and $\alpha\in\overline{\text{Im}(\intd)}$ continuous.

Let $f_{n}$ be a sequence of $C^{\infty}-$maps such that $f_{n}\to f$ in the $C^{1}-$topology, see \cite[Theorem 2.6]{Hirsch}. If $f_{n}$ is in the same path component as $f$ then $f_{n}$ and $f$ are homotopic so they induce the same map on cohomology. So we may assume without loss of generality that $H^{k}(f_{n}) = H^{k}(f)$. Let $\omega,\eta\in\mathcal{H}^{k}$ be such that
\begin{align*}
H^{k}(f)([\omega]) = [\eta]
\end{align*}
or since $H^{k}(f) = H^{k}(f_{n})$
\begin{align*}
H^{k}(f_{n})([\omega]) = [\eta].
\end{align*}
Since the lemma holds for $C^{\infty}-$maps we have
\begin{align*}
Pf_{n}^{*}\omega = \eta
\end{align*}
or since $f_{n}^{*}$ preserve $\ker\intd = \mathcal{H}^{k}\oplus\text{Im}(\intd)$ we have
\begin{align*}
f_{n}^{*}\omega = \eta + \alpha_{n},\quad\alpha_{n}\in\text{Im}(\intd).
\end{align*}
If it holds that $f_{n}^{*}\omega\to f^{*}\omega$ uniformly then it follows that $\alpha_{n}$ converges to a continuous element in $\overline{\text{Im}(\intd)}$ since
\begin{align*}
\alpha_{n} = f_{n}^{*}\omega - \eta\in\text{Im}(\intd)
\end{align*}
so we're done. So it suffices to show that $f_{n}^{*}\omega\to f^{*}\omega$ uniformly. Since $X$ is compact it suffices to show that $f_{n}^{*}\omega\to f^{*}\omega$ uniformly in some chart about every point $x\in X$. Now, let
\begin{align*}
\psi_{i}:X\supset U_{i}\to B\subset\mathbb{R}^{n}, \quad i = 1,2
\end{align*}
be charts about $x\in X$ and $fx\in X$ where $B$ is the open unit ball in $\mathbb{R}^{n}$. By possibly making $U_{1}$ smaller and $n$ larger we may assume that $f_{n}(U_{1}),f(U_{1})\subset U_{2}$ and that $\psi_{i},\psi_{i}^{-1}$ are uniformly bounded with uniformly bounded derivatives. Since $f_{n}\to f$ in $C^{1}$ we have
\begin{align*}
\psi_{2}f_{n}\psi_{1}^{-1}\to\psi_{2}f\psi_{1}^{-1}, \quad D(\psi_{2}f_{n}\psi_{1}^{-1})\to D(\psi_{2}f\psi_{1}^{-1})
\end{align*}
where we may assume that this convergence is uniform by possibly letting $U_{1}$ be smaller. Let $h_{n},h:B\to B$ denote
\begin{align*}
h = \psi_{2}f\psi_{1}^{-1},\quad h_{n} = \psi_{2}f_{n}\psi_{1}^{-1}.
\end{align*}
Then it holds that $h_{n}\to h$ and $Dh_{n}\to Dh$ uniformly. Let $I = (i_{1},...,i_{k})$ be multiindex $1\leq i_{1} < ... < i_{k}\leq n$ and define
\begin{align*}
e_{I}^{*} = e_{i_{1}}^{*}\wedge...\wedge e_{i_{k}}^{*}
\end{align*}
where $e_{i} = (0,...,1,....,0)$ is a unit vector. We note that
\begin{align*}
& |(h^{*}e_{I}^{*} - h_{n}^{*}e_{I}^{*})(\nu_{1},...,\nu_{k})| = \\ = &
|e_{i_{1}}^{*}\left((D_{x}h - D_{x}h_{n})\nu_{1}\right)|\cdot...\cdot |e_{i_{k}}^{*}\left((D_{x}h - D_{x}h_{n})\nu_{k}\right)|\leq \\ \leq &
\sup_{x}\norm{D_{x}h - D_{x}h_{n}}^{k}\prod_{i = 1}^{k}\norm{\nu_{i}}\to 0,\quad n\to\infty
\end{align*}
where the convergence is uniform if $\norm{\nu_{j}} = 1$. Since $e_{I}^{*}$, for all $I$, form a frame for $\Lambda^{k}(\text{T}^{*}B)$ it follows that $h_{n}^{*}\eta\to h^{*}\eta$ uniformly for every bounded $k-$form $\eta:B\to\Lambda^{k}(\text{T}^{*}B)$. In particular it holds for the section $(\psi_{2}^{-1})^{*}\omega:B\to\Lambda^{k}(\text{T}^{*}B)$ that
\begin{align*}
h_{n}^{*}\left(\psi_{2}^{-1}\right)^{*}\omega\to h^{*}\left(\psi_{2}^{-1}\right)^{*}\omega
\end{align*}
uniformly, but
\begin{align*}
& h_{n}^{*}\left(\psi_{2}^{-1}\right)^{*}\omega = \left(\psi_{1}^{-1}\right)^{*}f_{n}^{*}\psi_{2}^{*}\left(\psi_{2}^{-1}\right)^{*}\omega = \left(\psi_{1}^{-1}\right)^{*}f_{n}^{*}\omega, \\
& h^{*}\left(\psi_{2}^{-1}\right)^{*}\omega = \left(\psi_{1}^{-1}\right)^{*}f^{*}\psi_{2}^{*}\left(\psi_{2}^{-1}\right)^{*}f^{*}\omega = \left(\psi_{1}^{-1}\right)^{*}f^{*}\omega
\end{align*}
or since $\psi_{1}^{*}:\Omega^{k}(B)\to\Omega^{k}(U_{1})$ is an isomorphism we have that $f_{n}^{*}\omega\to f^{*}\omega$ uniformly on $U_{1}$.

Similarly we see that $f^{*}$ preserve $\overline{\ker(\intd)}$ and $\overline{\text{Im}(\intd)}$ since this holds for $f_{n}^{*}$.
\end{proof}
Let $\text{T}^{\mathbb{C}}X$ be the complexification of the tangent bundle with hermitian metric induced by the Riemannian metric. We define the space of complex $k-$forms, denoted $\Omega_{\mathbb{C}}^{k}(X)$, by the same construction as for real $k-$forms but using $\text{T}^{\mathbb{C}}X$. As in the real case we define $L^{2}\Omega_{\mathbb{C}}^{k}(X)$. We can define the pullback on complex $k-$forms by extending it from real $k-$forms and defining it to be complex linear. The laplacian on $\Omega_{\mathbb{C}}^{k}(X)$ is defined by extending the real laplacian to be complex linear. The space of complex harmonic forms, denoted $\mathcal{H}^{k}_{\mathbb{C}}$, is given by 
\begin{align*}
\mathcal{H}^{k}_{\mathbb{C}} = \mathcal{H}^{k}\oplus i\mathcal{H}^{k} = \mathcal{H}^{k}\otimes\mathbb{C}.
\end{align*}
We note that if $e^{\lambda}\in\mathbb{C}$ is an eigenvalue of $H^{k}(f):H^{k}(X)\to H^{k}(X)$ then we have a $\omega\in\mathcal{H}^{k}_{\mathbb{C}}$ such that $f^{*}\omega = \lambda\omega + \alpha$ where $\alpha\in\overline{\text{Im}(\intd)}\subset L^{2}\Omega_{\mathbb{C}}^{k}(X)$ is a continuous complex $k-$form. That is, when we complexify every eigenvalue of $H^{k}(f)$ has an eigenvector.

%% file: GeneralProofs.tex
In this section we prove the main result of the paper. We begin by framing the problem of finding bounds for the spectral radius as an equivalent question about finding non-trivial solutions to an equation, see $\eqref{MainEquation}$. We then study the solutions of equation $\eqref{MainEquation}$.

For the remainder of this section, let $(X,g)$ be a compact, oriented $n-$dimensional Riemannian manifold without boundary and let $f:X\to X$ be a $C^{1}-$map. We denote by $H^{k}(f):H^{k}(X)\to H^{k}(X)$ the induced map on the $k'$th cohomology group.

If $e^{\lambda}\in\mathbb{C}$ is a eigenvalue for $H^{k}(f)$ then we can find some harmonic $\omega\in\mathcal{H}_{\mathbb{C}}^{k}$ such that
\begin{align}
\label{MainEquation}
\tag{$\text{Eq}_{\lambda,k}$}
f^{*}\omega = e^{\lambda}\omega + \alpha
\end{align}
for some continuous $\alpha\in\overline{\text{Im}(\intd)}$. It follows that any eigenvalue of $H^{k}(f)$ implies a non-trivial solution of $\eqref{MainEquation}$. On the other hand we recall Definition $2.1$
\begin{definition*}
We say that $\omega,\alpha\in L^{2}\Omega^{k}_{\mathbb{C}}(X)$ is a solution of $\eqref{MainEquation}$ if $\omega\in\mathcal{H}^{k}_{\mathbb{C}}$, $\alpha\in\overline{\text{Im}(\intd)}$ and $\omega,\alpha$ satisfy $\eqref{MainEquation}$.
\end{definition*}
\begin{remark}
Since $\mathcal{H}_{\mathbb{C}}^{k}$ only contains smooth $k-$forms it follows that any solution $\omega,\alpha\in L^{2}\Omega^{k}(X)$ of $\eqref{MainEquation}$ satisfies that $\alpha$ is continuous.
\end{remark}
With this definition there is a one-to-one correspondence between the eigenvalues of $H^{k}(f)$ and the non-trivial solutions of $\eqref{MainEquation}$. It follows that we can bound the spectral radius of $H^{k}(f)$ by bounding the non-trivial solutions of $\eqref{MainEquation}$.
\begin{lemma}
Let $\omega,\alpha\in L^{2}\Omega_{\mathbb{C}}^{k}(X)$ be a solution of $\eqref{MainEquation}$. Then there exists a continuous sequence $\alpha_{n}\in\overline{\text{Im}(\intd)}$ such that
\begin{align*}
\left(f^{n}\right)^{*}\omega = e^{n\lambda}\omega + \alpha_{n}
\end{align*}
where $\alpha_{n}$ is given by
\begin{align*}
\alpha_{n} = e^{(n-1)\lambda}\sum_{j = 0}^{n-1}e^{-j\lambda}\left(f^{j}\right)^{*}\alpha
\end{align*}
\end{lemma}
\begin{proof}
We define $\alpha_{n}$ by
\begin{align*}
\alpha_{n} := \left(f^{n}\right)^{*}\omega - e^{n\lambda}\omega.
\end{align*}
Since $f^{*}$ preserve $\overline{\text{Im}(\intd)}$ by Lemma $3.1$ the lemma follows by showing that $\alpha_{n}$ satisfy the formula from the lemma. We note that for $n = 1$ the formula holds since $\omega,\alpha$ is a solution of $\eqref{MainEquation}$. So, we assume that the formula holds for some $n\geq 1$ and have
\begin{align*}
\alpha_{n+1} = & \left(f^{n+1}\right)^{*}\omega - e^{(n+1)\lambda}\omega = f^{*}\left(e^{n\lambda}\omega + e^{(n-1)\lambda}\sum_{j=0}^{n-1}e^{-j\lambda}\left(f^{j}\right)^{*}\alpha\right) - \\ & -
e^{(n+1)\lambda}\omega = \\ = &
e^{n\lambda}\left(f^{*}\omega - e^{\lambda}\omega + e^{-\lambda}\sum_{j = 0}^{n-1}e^{-j\lambda}\left(f^{j+1}\right)^{*}\alpha\right) = \\ = &
e^{n\lambda}\left(\alpha + e^{-\lambda}\sum_{j=1}^{n}e^{-(j-1)\lambda}\left(f^{j}\right)^{*}\alpha\right) = \\ = &
e^{n\lambda}\sum_{j = 0}^{n}e^{-j\lambda}\left(f^{j}\right)^{*}\alpha
\end{align*}
and the formula for $\alpha_{n}$ follows for all $n\geq1$ by induction.
\end{proof}
From this we immediately obtain estimates of $\text{Re}(\lambda)$ in terms of the growth rate, which is essentially contained in \cite{Yomdin1,KozlovskiIntForm}
\begin{lemma}
Let $\omega,\alpha\in L^{2}\Omega_{\mathbb{C}}^{k}(X)$ be solutions of $\eqref{MainEquation}$. If $\omega\neq0$ then
\begin{align*}
\text{Re}(\lambda)\leq\liminf_{n\to\infty}\frac{1}{n}\log\left(\int_{X}\norm{\left(D_{x}f^{n}\right)^{\wedge k}}\intd V_{g}(x)\right)
\end{align*}
\end{lemma}
\begin{proof}
From Lemma $4.1$ we have
\begin{align*}
e^{n\lambda} = \langle\left(f^{n}\right)^{*}\omega,\omega\rangle = \int_{X}\langle\left(f^{n}\right)^{*}\omega_{x},\omega_{x}\rangle\intd V_{g}(x)
\end{align*}
so by the Cauchy-Schwartz inequality
\begin{align*}
e^{n\text{Re}(\lambda)}\leq\int_{X}\norm{\left(f^{n}\right)^{*}\omega_{x}}\norm{\omega_{x}}\intd V_{g}(x)\leq\norm{\omega}_{C^{0}}\int_{X}\norm{\left(f^{n}\right)^{*}\omega_{x}}\intd V_{g}(x).
\end{align*}
Let $\nu\in\text{T}_{x}X\wedge...\wedge\text{T}_{x}X$ be a $k-$vector then
\begin{align*}
\norm{\left(f^{n}\right)^{*}\omega_{x}} = & \sup_{\norm{\nu} = 1}\left|\left(f^{n}\right)^{*}\omega_{x}(\nu)\right| = \sup_{\norm{\nu} = 1}\left|\omega_{f^{n}x}\left(\left(D_{x}f^{n}\right)^{\wedge k}\nu\right)\right|\leq\\ \leq &
\norm{\left(D_{x}f^{n}\right)^{\wedge k}}\norm{\omega}_{C^{0}}.
\end{align*}
Combining these formulas we have
\begin{align*}
\text{Re}(\lambda)\leq\frac{2\log\norm{\omega}_{C^{0}}}{n} + \frac{1}{n}\log\int_{X}\norm{\left(D_{x}f^{n}\right)^{\wedge k}}\intd V_{g}(x)
\end{align*}
and by taking the $\liminf$ on both side the lemma follows.
\end{proof}

%% file: ProofOfTheoremAndCorollaryA.tex
In this section we prove Theorem $A$ and Corollary $A$. We begin by proving Theorem $A$, which follows from Lemma $4.2$ and the fact that
\begin{align*}
\lim_{n\to\infty}\frac{1}{n}\sup_{x}\log\norm{\left(D_{x}f^{n}\right)^{\wedge k}} = & \lambda^{+}(Df^{\wedge k}) = \sup_{\mu}\lambda^{+}(Df^{\wedge k},\mu) = \\ = & \sup_{\mu}\Lambda_{k}(Df,\mu).
\end{align*}
Indeed, for any $\varepsilon>0$ and $n\geq n_{0}(\varepsilon)$ we have
\begin{align*}
\norm{\left(D_{x}f^{n}\right)^{\wedge k}} = e^{n\left(\frac{1}{n}\log\norm{\left(D_{x}f^{n}\right)^{\wedge k}}\right)}\leq e^{n(\lambda^{+}(Df^{\wedge k}) + \varepsilon)}
\end{align*}
so it follows that
\begin{align*}
& \liminf_{n\to\infty}\frac{1}{n}\log\int_{X}\norm{\left(D_{x}f^{n}\right)^{\wedge k}}\intd V_{g}(x)\leq \\ \leq & \liminf_{n\to\infty}\frac{1}{n}\log\int_{X}e^{n(\lambda^{+}(Df^{\wedge k}) + \varepsilon)}\intd V_{g}(x) = \lambda^{+}(Df^{\wedge k}) + \varepsilon
\end{align*}
so letting $\varepsilon\to0$ Theorem $A$ follows from Lemma $4.2$.

%% file: ProofOfTheoremAndCorollaryB.tex
In this section we prove Theorem $B$ and Corollary $B$. In the remainder of this section we shall assume that $f:X\to X$ preserves a probability measure equivalent to volume. To simplify notation we shall denote the Lyapunov exponents with respect to the invariant volume by $\lambda_{i}(x) := \lambda_{i}(x,Df,V)$. So we have
\begin{align*}
\lambda_{1}(x)\geq\lambda_{2}(x)\geq...\geq\lambda_{\dim(X)}(x)
\end{align*}
and
\begin{align*}
\Lambda_{k}(x) = \sum_{i = 1}^{k}\lambda_{i}(x).
\end{align*}
We begin by giving a equivalent condition for integrability of a metric $h:X\to\text{T}^{*}X\otimes\text{T}^{*}X$. Let $\mathcal{E}\to X$ be a continuous metric vector bundle over $X$ of rank $r$ and with metric $g$. Furthermore let $F:\mathcal{E}\to\mathcal{E}$ be a cocycle over $f:X\to X$. We note that there always exists a measurable global $g-$orthonormal frame of $\mathcal{E}$, which can be defined in charts and then glued together with discontinuities where the different charts meet. For a $g-$orthonormal frame $e_{1},...,e_{r}\in\Gamma(\mathcal{E})$ and a metric $h:X\to\mathcal{E}^{*}\otimes\mathcal{E}^{*}$ we define
\begin{align*}
h_{ij}(x) := h_{x}(e_{i}(x),e_{j}(x)):X\to\mathbb{R}.
\end{align*}
Using Einsteins summation convention we can calculate the norm of $h_{x}\in\mathcal{E}_{x}^{*}\otimes\mathcal{E}_{x}^{*}$ in terms of $h_{ij}$
\begin{align*}
\norm{h_{x}}^{2} = & \norm{h_{ij}(x)e^{i}(x)\otimes e^{j}(x)}^{2} = \\ = &
\sum_{i,j}|h_{ij}(x)|^{2}g(e^{i}(x),e^{i}(x))g(e^{j}(x),e^{j}(x)) = \sum_{i,j}|h_{ij}(x)|^{2}
\end{align*}
where $e^{i}$ is the dual element of $e_{i}$. Before stating the next lemma, we say that a function (or more generally a section of some metric vector bundle) $\sigma$ is $L^{p}$ with $p\leq 1$ if
\begin{align*}
\int_{X}\norm{\sigma}^{p}\intd V_{g} < \infty.
\end{align*}
For $p < 1$ the integral above is not a norm.
\begin{lemma}
The metric $h$ is $L^{p}$ if and only if each $h_{ij}$ is $L^{p}$ (where we allow $p < 1$). 
\end{lemma}
\begin{proof}
Let $1\leq i,j\leq r$ then we can bound the $L^{p}-$norm of $h_{ij}$ as
\begin{align*}
\norm{h_{ij}}_{L^{p}}^{p} = & \int_{X}|h_{ij}|^{p}\intd V_{g} \leq \int_{X}\left(\sum_{i,j}|h_{ij}|^{2}\right)^{p/2}\intd V_{g} = \\ = &
\int_{X}\norm{h_{x}}^{p}\intd V_{g} = \norm{h}_{L^{p}}^{p}
\end{align*}
so $h_{ij}\in L^{p}(X)$ if $h\in L^{p}(\mathcal{E}^{*}\otimes\mathcal{E}^{*})$. On the other hand, if each $h_{ij}\in L^{p}(X)$ then
\begin{align*}
\norm{h}_{L^{p}}^{p} = & \int_{X}\norm{h_{x}}^{p}\intd V_{g}(x) = \int_{X}\left(\sum_{i,j}|h_{ij}(x)|^{2}\right)^{p/2}\intd V_{g}(x) = \\= &
\int_{X}r^{p}\left(\max_{i,j}|h_{ij}(x)|^{2}\right)^{p/2}\intd V_{g}(x) = r^{p}\int_{X}\max_{i,j}|h_{ij}(x)|^{p}\intd V_{g}(x)
\end{align*}
since each $h_{ij}$ is in $L^{p}$ then $\max_{ij}|h_{ij}|$ is also in $L^{p}$, and it follows that $h$ is $L^{p}$.
\end{proof}
\begin{lemma}
If $h$ is a $L^{p}-$metric on $\text{T}^{\mathbb{C}}X$, then $h$ induces a $L^{p/k}-$metric on $\Lambda^{k}(\text{T}^{\mathbb{C}}X)$.
\end{lemma}
\begin{proof}
The induced metric on $\Lambda^{k}(\text{T}^{\mathbb{C}}X)$ is given by
\begin{align*}
h_{x}(v_{1}\wedge...\wedge v_{k},w_{1}\wedge...\wedge w_{k}) = \det(h(v_{i},w_{j}))
\end{align*}
for decomposable vectors $v_{1}\wedge...\wedge v_{k},w_{1}\wedge...\wedge w_{k}\in\Lambda^{k}(\text{T}_{x}^{\mathbb{C}}X)$. We recall that if $e_{i}\in\text{T}_{x}X$ is a $g-$orthonormal basis then
\begin{align*}
\{e_{I} = e_{i_{1}}\wedge...\wedge e_{i_{k}}\text{ : }I = (i_{1},...,i_{k}),\text{ }1\leq i_{1} < ... < i_{k}\leq n\}
\end{align*}
is a orthonormal basis of $\Lambda^{k}(\text{T}X)$. Let $e_{i}:X\to\text{T}X$ be a (not necessarily continuous) $g-$orthonormal frame and let $e_{I}:X\to\Lambda^{k}(\text{T}X)$ be the corresponding orthonormal frame of $\Lambda^{k}(\text{T}X)$. By Lemma $4.3$ it suffices to show that each $h_{IJ} = \det(h(e_{i_{k}}(x),e_{j_{\ell}}(x)))$ is $L^{p/k}$. But $\det(h(e_{i_{k}}(x),e_{j_{\ell}}(x)))$ is a homogeneous polynomial of degree $k$ in the variables $h_{ij}$, $1\leq i,j\leq n$. Since each $h_{ij}$ is in $L^{p}$ by Lemma $4.3$, it suffices to show that if $f_{1},...,f_{k}\in L^{p}(X)$ then $f_{1}\cdot...\cdot f_{k}\in L^{p/k}(X)$. This follows from Hölders inequality
\begin{align*}
\int_{X}|f_{1}|^{p/k}\cdot...\cdot|f_{k}|^{p/k}\intd V_{g}\leq\left(\int_{X}|f_{1}|^{p}\intd V_{g}\right)^{1/k}\cdot...\cdot\left(\int_{X}|f_{k}|^{p}\intd V_{g}\right)^{1/k}.
\end{align*}
\end{proof}
\begin{lemma}
If $g$ and $h$ are inner products on some finite dimensional vector space $V$ and $g(u,u)\leq C\cdot h(u,u)$, $C>0$, for all $u\in V$, then the induced inner products, $g^{k},h^{k}$, on $\Lambda^{k}(V)$ also satisfy $g^{k}(w,w)\leq C^{k}\cdot h^{k}(w,w)$ for $w\in\Lambda^{k}(V)$.
\end{lemma}
\begin{proof}
After possibly changing $g$ to $g/C$ we may assume without loss of generality that $C=1$.

Let $e_{i}\in V$ be a $g-$orthonormal basis and a $h-$orthogonal basis. Such a basis always exists since $h(u,v) = g(Qu,v)$ for some positive and $g-$self-adjoint $Q:V\to V$, so there exists a $g-$orthonormal basis of eigenvectors for $Q$. This basis is then also orthogonal for $h$. Let 
\begin{align*}
S_{k} := \{I = (i_{1},...,i_{k})\text{ : }1\leq i_{1} < ... < i_{k}\leq\dim(V)\}
\end{align*}
and for $I\in S_{k}$ we define $e_{I} = e_{i_{1}}\wedge...\wedge e_{i_{k}}$. Then $\{e_{I}\text{ : }I\in S_{k}\}$ forms a basis for $\Lambda^{k}(V)$. This basis is orthonormal with respect to $g^{k}$ and orthogonal with respect to $h^{k}$, which follows since $e_{i}$ is a orthonormal basis for $g$ and a orthogonal basis for $h^{k}$. It follows from the Pythagorean theorem that
\begin{align*}
\norm{u}_{h^{k}}^{2} = \norm{u^{I}e_{I}}_{h^{k}}^{2} = \norm{\left(u^{I}\norm{e_{I}}\right)\frac{e_{I}}{\norm{e_{I}}}}_{h^{k}}^{2} = \sum_{I\in S_{k}}\left|u_{I}\right|^{2}\norm{e_{I}}_{h^{k}}^{2}
\end{align*}
but since $g(u,u)\leq h(u,u)$ for $u\in V$ and the basis $e_{i}$ is $h-$orthogonal we have
\begin{align*}
\norm{e_{I}}_{h^{k}}^{2} = \det(h(e_{i},e_{j})) = \norm{e_{i_{1}}}_{h}^{2}\cdot...\cdot\norm{e_{i_{k}}}_{h}^{2}\geq\norm{e_{i_{1}}}_{g}^{2}\cdot...\cdot\norm{e_{i_{k}}}_{g}^{2} = 1
\end{align*}
so we have
\begin{align*}
\norm{u}_{h^{k}}^{2}\geq\sum_{I\in S_{k}}|u^{I}|^{2} = \norm{u}_{g^{k}}^{2}.
\end{align*}
\end{proof}
Let $h^{\varepsilon}:X\to\text{T}^{*}X\otimes\text{T}^{*}X$ be the Lyapunov metric defined by
\begin{align*}
h^{\varepsilon}(u,v) := \sum_{i}h_{i}^{\varepsilon}(u,v)
\end{align*}
where $h_{i}^{\varepsilon}$ is the inner product defined on $H_{i}(x)$ by
\begin{align*}
h_{i}^{\varepsilon}(u,v) := \sum_{n\in\mathbb{Z}}e^{-2n|\varepsilon|}e^{-2n\Tilde{\lambda}_{i}(x)}\left(f^{n}\right)^{*}g(u,v)
\end{align*}
where $\Tilde{\lambda}_{i}(x)$ is the Lyapunov exponent associated to $H_{i}(x)$. We note that $h^{\varepsilon}$ is measurable and $V_{g}-$almost everywhere defined. Let $h^{\varepsilon,k}$ be the metric on $\Lambda^{k}(\text{T}X)$ induced by $h^{\varepsilon}$. We recall the standard fact, see for example \cite{KatokIntroDynSyst}, that for $u\in H_{i}(x)$ the Lyapunov metric $h^{\varepsilon}$ satisfy
\begin{align*}
\norm{u}_{h^{\varepsilon}}^{2}e^{2n(\lambda_{i}(x) - \varepsilon)}\leq h^{\varepsilon}\left(Df^{n}(u),Df^{n}(u)\right)\leq\norm{u}_{h^{\varepsilon}}^{2}e^{2n(\lambda_{i}(x) + \varepsilon)}.
\end{align*}
We want to extend this to the metric $h^{\varepsilon,k}$.
\begin{lemma}
The metric $h^{\varepsilon,k}$ satisfies
\begin{align*}
\norm{(D_{x}f^{n})^{\wedge k}}_{h^{\varepsilon,k}}^{2}\leq Ce^{2n(\Lambda_{k}(x) + k\varepsilon)}
\end{align*}
where $C$ is a constant that only depends on the manifold $X$.
\end{lemma}
\begin{proof}
Let $e_{i,1},...,e_{i,u_{i}(x)}\in H_{i}(x)$, $\dim H_{i}(x) = u_{i}(x)$, be a $h^{\varepsilon}-$orthonormal basis. We note that for any $1\leq i_{1} < ... < i_{\ell}\leq u_{i}(x)$ we have
\begin{align*}
& \norm{\left(D_{x}f^{n}\right)^{\wedge\ell}e_{i,i_{1}}\wedge...\wedge e_{i,i_{\ell}}}_{h^{\varepsilon,k}}^{2} = \det\left(h^{\varepsilon}\left(D_{x}f^{n}(e_{i,i_{a}}),D_{x}f^{n}(e_{i,i_{b}})\right)\right) = \\ = &
\sum_{\sigma\in S_{\ell}}\text{sgn}(\sigma)\prod_{j = 1}^{\ell}h^{\varepsilon}(D_{x}f^{n}(e_{i,i_{j}}),D_{x}f^{n}(e_{i,i_{\sigma(j)}}))\leq \\ \leq &
\sum_{\sigma\in S_{\ell}}\prod_{j = 1}^{\ell}\norm{D_{x}f^{n}(e_{i,i_{j}})}_{h^{\varepsilon}}\norm{D_{x}f^{n}(e_{i,i_{\sigma(j)}})}_{h^{\varepsilon}}\leq \\ \leq &
(\ell!)e^{n\ell(\lambda_{i}(x) + \varepsilon)}e^{n\ell(\lambda_{i}(x) + \varepsilon)} = (\ell!)e^{2n\ell(\lambda_{i}(x) + \varepsilon)}
\end{align*}
where $S_{n}$ is the permutation group of $n$ elements. Now let $\ell_{1},...,\ell_{k(x)}\geq0$ be such that $\ell_{1}+...+\ell_{k(x)} = \ell$. Consider $E_{i} = e_{i,q_{i,1}}\wedge...\wedge e_{i,q_{i,\ell_{i}}}$, or $E_{i} = 1$ if $\ell_{i} = 0$, and $E = E_{1}\wedge...\wedge E_{k(x)}$. We denote by $d_{1} = e_{1,q_{1,1}}$, $d_{2} = e_{1,q_{1,2}}$, $d_{\ell_{1} + 1} = e_{2,q_{2,1}}$ and so forth until $d_{\ell} = e_{k(x),q_{k(x),\ell_{k(x)}}}$. That is, $d_{i}$ are chosen such that $d_{1},...,d_{\ell_{1}}\in H_{1}(x)$ are orthonormal, $d_{\ell_{1}+1},...,d_{\ell_{1} + \ell_{2}}\in H_{2}(x)$ are orthonormal and so forth until $d_{\ell_{1}+...+\ell_{k(x)-1} + 1},...,d_{\ell}\in H_{k(x)}(x)$ are orthonormal. So we can write $E = d_{1}\wedge...\wedge d_{\ell}$. Since the spaces $H_{i}(x)$ are orthogonal and since $D_{x}f^{n}(H_{i}(x)) = H_{i}(f^{n}x)$ we have that the matrix
\begin{align*}
A_{ij} := h^{\varepsilon}\left(D_{x}f^{n}(d_{i}),D_{x}f^{n}(d_{j})\right)
\end{align*}
is a block matrix such that
\begin{align*}
A = \begin{pmatrix}
A_{1} & & & \\
& A_{2} & & \\
& & \ddots & \\
& & & A_{k(x)}
\end{pmatrix}
\end{align*}
where each $A_{j}$ is a $\ell_{j}\times\ell_{j}-$matrix given by
\begin{align*}
\left(A_{j}\right)_{ab} = h^{\varepsilon}\left(D_{x}f^{n}(d_{\ell_{1}+...+\ell_{j-1} + a}),D_{x}f^{n}(d_{\ell_{1}+...+\ell_{j-1} + b}),\right)
\end{align*}
so in particular we have from the calculation above that
\begin{align*}
\det(A_{j}) = \norm{\left(D_{x}f^{n}\right)^{\wedge\ell_{j}}E_{j}}_{h^{\varepsilon,k}}^{2}\leq(\ell_{j}!)e^{2n\ell_{j}(\lambda_{j}(x) + \varepsilon)}.
\end{align*}
Now we can calculate the norm of $\left(D_{x}f^{n}\right)^{\wedge\ell}E$ as
\begin{align*}
\norm{\left(D_{x}f^{n}\right)^{\wedge\ell}E}_{h^{\varepsilon,k}}^{2} = & \det(A) = \prod_{j = 1}^{k(x)}\det\left(A_{j}\right)\leq\prod_{j = 1}^{k(x)}(\ell_{j}!)e^{2n\ell_{j}\left(\lambda_{j}(x) + \varepsilon\right)} = \\ = &
\left(\prod_{j=1}^{k(x)}\ell_{j}!\right)e^{2n\left(\sum_{j=1}^{k(x)}\ell_{j}\lambda_{j}(x) + \ell\varepsilon\right)}\leq \\ \leq &
\left((\dim X)!\right)^{\dim X}e^{2n(\Lambda_{\ell}(x) + \ell\varepsilon)}
\end{align*}
where the last inequality follows since $\ell_{1}\lambda_{1}(x) + ... +\ell_{k(x)}\lambda_{k(x)}$ is a sum of $\ell$ Lyapunov exponents which is in particular smaller then the sum of the $\ell$ largest Lyapunov exponents. So let $C^{2} = \left(\dim X!\right)^{\dim X}$. Since $e_{i,\ell}\in\text{T}_{x}X$, where $i = 1,...,k(x)$ and $\ell = 1,...,u_{i}(x)$, forms a orthonormal basis of $\text{T}_{x}X$. If we order the elements $e_{i,\ell}$ as $d_{1},...,d_{\dim(X)}\in\text{T}_{x}X$ then
\begin{align*}
\{d_{I} = d_{i_{1}}\wedge...\wedge d_{i_{\ell}}\text{ : }I = (i_{1},...,i_{\ell}),\text{ }1\leq i_{1} < ... < i_{\ell}\leq\dim(X)\}
\end{align*}
forms a $h^{\varepsilon,k}-$orthonormal basis of $\Lambda^{\ell}(\text{T}_{x}X)$, and from the calculation above it satisfy
\begin{align*}
\norm{\left(D_{x}f^{n}\right)^{\wedge\ell}d_{I}}_{h^{\varepsilon,k}}\leq Ce^{n(\Lambda_{k}(x) + \ell\varepsilon)}.
\end{align*}
Since for any $u\in\Lambda^{\ell}(\text{T}_{x}X)$ with $\norm{u}_{h^{\varepsilon}} = 1$ we have
\begin{align*}
\norm{u}_{h^{\varepsilon,k}}^{2} = \norm{\sum_{I}u_{I}d_{I}}_{h^{\varepsilon}}^{2} = \sum_{I}|u_{I}|^{2} = 1
\end{align*}
so each $|u_{I}|\leq 1$ and we have
\begin{align*}
\norm{\left(D_{x}f^{n}\right)^{\wedge\ell}u}_{h^{\varepsilon,k}}\leq & \sum_{I}|u_{I}|\norm{\left(D_{x}f^{n}\right)^{\wedge\ell}d_{I}}_{h^{\varepsilon,k}}\leq \\ \leq & C\dim\left(\Lambda^{\ell}(\text{T}_{x}X)\right)e^{n(\Lambda_{\ell}(x) + \ell\varepsilon)}\leq C'e^{n(\Lambda_{\ell}(x) + \ell\varepsilon)}
\end{align*}
and the Lemma follows.
\end{proof}
We can now prove Theorem $B$. Let $h:X\to\mathcal{E}^{*}\otimes\mathcal{E}^{*}$ be a metric on a metric vector bundle $\mathcal{E}\to X$ with metric $g$ and rank $r$. Let $e_{i}\in\mathcal{E}_{x}$ be a $g-$orthonormal basis and $u\in\mathcal{E}_{x}$ a unit vector. If we denote by $e^{i}$ the dual basis of $e_{i}$ we have
\begin{align*}
h = \sum_{i,j}h_{ij}e^{i}\otimes e^{j},\quad h_{ij} = h(e_{i},e_{j})
\end{align*}
and we can calculate
\begin{align*}
\norm{h}_{g}^{2} = & \sum_{i,j,k,\ell}h_{ij}\overline{h}_{k\ell}g\left(e^{i}\otimes e^{j},e^{k}\otimes e^{\ell}\right) = \\ = &
\sum_{i,j,k,\ell}h_{ij}\overline{h}_{k\ell}\delta^{ik}\delta^{j\ell} = \sum_{i,j}h_{ij}\overline{h}_{ij} = \sum_{i,j}|h_{ij}|^{2}
\end{align*}
where $\delta^{ab}$ is the Kronecker delta defined by $\delta^{ab} = 1$ if $a = b$ and $\delta^{ab} = 0$ if $a\neq b$. Then we have
\begin{align*}
\norm{u}_{h}^{2} = & \norm{\sum_{i=1}^{r}u_{i}e_{i}}_{h}^{2} = \sum_{i,j}u_{i}\overline{u}_{j}h(e_{i},e_{j})\leq \sum_{i,j}r^{2}\max_{j}|u_{j}|^{2}|h_{ij}|\leq \\ \leq & \sum_{i,j}r^{2}\norm{u}_{g}^{2}|h_{ij}|\leq r^{4}\norm{u}_{g}^{2}\left(\max_{i,j}|h_{ij}|^{2}\right)^{1/2}\leq \\ \leq &
r^{4}\norm{u}_{g}^{2}\norm{h}_{g}
\end{align*}
so $\norm{u}_{h}\leq C\norm{u}_{g}\norm{h}_{g}^{1/2}$ for some constant $C$ that only depends on the rank of $\mathcal{E}$. Let $g$ on $X$, and let $h^{\varepsilon}$ be the Lyapunov metric, and let $g^{k}$, $h^{\varepsilon,k}$ be the induced metrics on $\Lambda^{k}(\text{T}X)$. It's clear that $h^{\varepsilon}(u,u)\geq g(u,u)$ for $u\in H_{i}(x)$ when $H_{i}(x)$ and $h^{\varepsilon}$ is defined. For $u\in\text{T}_{x}X$ let $u_{i}$ be the projection onto $H_{i}(x)$. Using the Cauchy-Schwartz inequality we obtain
\begin{align*}
g(u,u) = & \sum_{i,j}g(u_{i},u_{j})\leq\sum_{i,j}\norm{u_{i}}_{g}\norm{u_{j}}_{g}\leq
\left(k(x)\right)^{2}\sum_{i}g(u_{i},u_{i})\leq \\ \leq & \left(\dim(X)\right)^{2}h^{\varepsilon}(u,u)
\end{align*}
so we have $g(u,u)\leq C\cdot h^{\varepsilon}(u,u)$ where $C$ is a constant that only depends on the manifold. It follows from Lemma $4.5$
\begin{align*}
\norm{\left(D_{x}f^{n}\right)^{\wedge k}}_{g^{k}} = & \sup_{\norm{u}_{g^{k}} = 1}\norm{\left(D_{x}f^{n}\right)^{\wedge k}(u)}_{g^{k}}\leq \\ \leq & 
C^{k}\sup_{\norm{u}_{g^{k}} = 1}\norm{\left(D_{x}f^{n}\right)^{\wedge k}(u)}_{h^{\varepsilon,k}} = \\ = &
C^{k}\sup_{\norm{u}_{g^{k}} = 1}\norm{u}_{h^{\varepsilon,k}}\norm{\left(D_{x}f^{n}\right)^{\wedge k}\left(\frac{u}{\norm{u}_{h^{\varepsilon,k}}}\right)}_{h^{\varepsilon,k}}
\end{align*}
using Lemma $4.6$ and the calculation above we have a constant $L$ such that
\begin{align*}
\norm{\left(D_{x}f^{n}\right)^{\wedge k}}_{g^{k}}\leq & L\sup_{\norm{u}_{g^{k}} = 1}\norm{u}_{h^{\varepsilon,k}}e^{n(\Lambda_{k}(x) + k\varepsilon)}\leq \\ \leq &
\sup_{\norm{u}_{g^{k}} = 1}LC'\norm{u}_{g^{k}}\norm{h^{\varepsilon,k}}_{g^{k}}^{1/2}e^{n(\Lambda_{k}(x) + k\varepsilon)} = \\ = &
C''\norm{h^{\varepsilon,k}}_{g^{k}}^{1/2}e^{n(\Lambda_{k}(x) + k\varepsilon)}.
\end{align*}
If we denote by $\norm{\Lambda_{k}(x)}_{L^{\infty}}$ the essential supremum of $\Lambda_{k}(x)$ then it follows that
\begin{align*}
& \frac{1}{n}\log\int_{X}\norm{\left(D_{x}f^{n}\right)^{\wedge k}}_{g}\intd V_{g}(x)\leq \\ \leq & 
\frac{1}{n}\log C''\int_{X}\norm{h^{\varepsilon,k}_{x}}_{g^{k}}^{1/2}e^{n(\Lambda_{k}(x) + k\varepsilon)}\intd V_{g}(x)\leq \\ \leq &
\frac{\log C''}{n} + \norm{\Lambda_{k}(x)}_{L^{\infty}} + k\varepsilon + \frac{1}{n}\log\int_{X}\norm{h_{x}^{\varepsilon,k}}_{g^{k}}^{1/2}\intd V_{g}(x).
\end{align*}
If $h^{\varepsilon}$ is $L^{k/2}$ then it follows from Lemma $4.4$ that $h^{\varepsilon,k}$ is in $L^{1/2}$ and it follows by letting $n\to\infty$ that
\begin{align*}
& \liminf_{n\to\infty}\frac{1}{n}\log\int_{X}\norm{\left(D_{x}f^{n}\right)^{\wedge k}}_{g}\intd V_{g}(x)\leq\norm{\Lambda_{k}(x)}_{L^{\infty}} + k\varepsilon
\end{align*}
which proves Theorem B.

%% file: ProofOfCorollaries.tex
In this section we prove all corollaries stated in section $2$. Corollary A follows immediately from Theorem A and the definition of uniform subexponential growth combined with the universal coefficients theorem. Indeed, by the universal coefficients theorem, see \cite{Hatcher}, we have a natural isomorphism
\begin{align*}
H^{k}(X) = \text{Hom}(H_{k}(X),\mathbb{R})
\end{align*}
so $H^{*}(f)$ can be interpreted as the dual map of $f_{*}$, so the maps share spectrum. The first part of Corollary B follows from Theorem A since the map $\mu\mapsto\Lambda_{k}(Df,\mu) = \lambda^{+}(Df^{\wedge k},\mu)$ is a upper semi-continuous, see \cite{BeyondUniformBonatti}, and therefore attains it's maximum. The second part of Corollary B follows by to passing from cohomology to homology and noting that $\Lambda_{k}(x)\leq\Sigma(x)$ for every $k$.

Corollary C follows from Corollary B. Indeed the first part is clear. The second part follows since if $f$ preserves a continuous volume form, then $\det(D_{x}f^{n})$ is uniformly bounded in $x$ and $n$ so we have
\begin{align*}
0 = \lim_{n\to\infty}\frac{1}{n}\log\det(D_{x}f^{n}) = \sum_{i = 1}^{\dim(X)}\lambda_{i}(x,Df,\mu)
\end{align*}
for every $f-$invariant measure $\mu$ by Oseledec's theorem. The first part of Corollary D follows from Theorem B and the universal coefficient theorem. The second part follows since $f$ preserves a volume, so the sum of Lyapunov exponents vanish. Corollary E follows from Pesin's entropy formula
\begin{align*}
h_{\mu}(f) = \int_{X}\Sigma(x;Df,\mu)\intd\mu(x)
\end{align*}
which can be applied for $f$ $C^{1+\alpha}$. Since $V$ is an ergodic measure $\Sigma(x;Df,V)$ is constant and it follows that $h_{V}(f) = \norm{\Sigma(x;Df,V)}_{L^{\infty}}$.

Finally, Corollary F follows by noting that the proof of Theorem B shows that
\begin{align*}
\liminf_{n\to\infty}\frac{1}{n}\log\int_{X}\norm{\left(D_{x}f^{n}\right)^{\wedge k}}\intd V_{g}(x)\leq\norm{\Sigma(x,Df,V)}_{L^{\infty}} + \dim(X)\varepsilon
\end{align*}
for every $\varepsilon>0$ (under the assumptions of Corollary F). Letting $\varepsilon\to 0$ and using that the Lyapunov exponents are constant almost everywhere we have
\begin{align*}
\liminf_{n\to\infty}\frac{1}{n}\log\int_{X}\norm{\left(D_{x}f^{n}\right)^{\wedge k}}\intd V_{g}(x)\leq\Sigma(Df,V)
\end{align*}
for every $k$. Let $(Df^{n})^{\wedge}$ be the exterior map of $f$ defined on the exterior algebra by
\begin{align*}
\left(D_{x}f^{n}\right)^{\wedge}:\Lambda(\text{T}_{x}X) = \bigoplus_{k = 0}^{\dim(X)}\Lambda^{k}(\text{T}_{x}X)\to\Lambda(\text{T}_{f^{n}x}X)
\end{align*}
Since $f$ is assumed to be a $C^{\infty}-$diffeomorphism we can use the main result from \cite{KozlovskiIntForm} and the Pesin formula to obtain
\begin{align*}
h_{\text{top}}(f) = & \lim_{n\to\infty}\frac{1}{n}\log\int_{X}\norm{\left(D_{x}f^{n}\right)^{\wedge}}\intd V_{g}(x) = \\ = &
\liminf_{n\to\infty}\frac{1}{n}\log\int_{X}\norm{\left(D_{x}f^{n}\right)^{\wedge}}\intd V_{g}(x)\leq \\ \leq &
\Sigma(Df,V) = h_{V}(f)\leq h_{\text{top}}(f)
\end{align*}
and the Corollary follows from the variational principle.

%% file: main.bbl
\begin{thebibliography}{10}

\bibitem{BeyondUniformBonatti}
C.~Bonatti, L.~J. Dìaz, and M.~Viana.
\newblock {\em Dynamics Beyond Uniform Hyperbolicity}.
\newblock Springer-Verlag, Berlin Heidelberg, 2010.

\bibitem{Bowen1}
R.~Bowen.
\newblock Entropy and the fundamental group.
\newblock In {\em The Structure of Attractors in Dynamical Systems}, Berlin
  Heidelberg, 1978. Springer-Verlag.

\bibitem{Hatcher}
A.~Hatcher.
\newblock {\em Algebraic topology}.
\newblock Cambridge University Press, 2002.

\bibitem{Hirsch}
M.~W. Hirsch.
\newblock {\em Differential Topology}.
\newblock Springer-Verlag, New York, 1976.

\bibitem{GeomAnalysis}
J.~Jost.
\newblock {\em Riemannian Geometry and Geometric Analysis}.
\newblock Springer International Publishing, 2017.

\bibitem{Katok1}
A.~Katok.
\newblock Lyapunov exponents, entropy and periodic orbits for diffeomorphisms.
\newblock {\em Publications Mathématiques de L’Institut des Hautes
  Scientifiques}, 51:137–173, 1980.

\bibitem{KatokIntroDynSyst}
A.~Katok and B.~Hasselblatt.
\newblock {\em Introduction to the Modern Theory of Dynamical Systems}.
\newblock Cambridge University Press, 1995.

\bibitem{KozlovskiIntForm}
O.~Kozlovski.
\newblock An integral formula for topological entropy of $c^{\infty}$ maps.
\newblock {\em Ergodic Theory and Dynamical Systems}, Volume 18:405--424, 1998.

\bibitem{Manning1}
A.~Manning.
\newblock Topological entropy and the first homology group.
\newblock In {\em Dynamical Systems—Warwick 1974}, Berlin Heidelberg, 1975.
  Springer-Verlag.

\bibitem{Marzantowicz1}
W.~Marzantowicz and F.~Przytycki.
\newblock Entropy conjecture for continuous maps of nilmanifolds.
\newblock {\em Isr. J. Math.}, 165:349–379, 2008.

\bibitem{MisiurewiczAndPrzytycki1}
M.~Misiurewicz and F.~Przytycki.
\newblock Topological entropy and degree of smooth mappings.
\newblock {\em Bull. Acad. Pol. Sci., Ser. sci. math., astr. et phys.},
  25:573--574, 1977.

\bibitem{LectureNotesPollicott}
M.~Pollicott.
\newblock {\em Lectures on ergodic theory and Pesin theory on compact
  manifolds}.
\newblock Cambridge University Press, 1993.

\bibitem{SchreiberSubadditiveFunction}
S.~J. Schreiber.
\newblock On growth rates of subadditive functions for semiflows.
\newblock {\em Journal of Differential Equations}, 148(2):334--350, 1998.

\bibitem{Shub1}
M.~Shub.
\newblock Dynamical systems, filtrations and entropy.
\newblock {\em Bull. Amer. Math. Soc.}, 80:27--41, 1974.

\bibitem{Yomdin1}
Y.~Yomdin.
\newblock Volume growth and entropy.
\newblock {\em Israel J. Math.}, 57:285--300, 1987.

\end{thebibliography}
